\documentclass{article}
\usepackage[utf8]{inputenc}

\usepackage{amsmath}
\usepackage{amssymb}
\usepackage{indentfirst}
\usepackage{authblk}
\usepackage{hyperref}
\usepackage{setspace}
\spacing{1.25}

\newtheorem{theorem}{Theorem}
\newtheorem{lemma}[theorem]{Lemma}
\newtheorem{prop}[theorem]{Proposition}
\newenvironment{proof}{\textit{Proof:}}{$\square$}

\newcommand{\pdiv}{\mid\!\mid}
\DeclareMathOperator{\rad}{rad}
\newcommand{\Mod}[1]{\ \left(\mathrm{mod}\ #1\right)}

\title{On the Radical of Multiperfect Numbers and Applications} 

\author{Viraj Jayam, Ajit Kadaveru, Nithin Kavi, and Xinyi Zhang}

\date{}

\begin{document}

\maketitle

\begin{abstract}

It is conjectured that for a perfect number $m,$ $\rad \left(m\right)\ll m^{\frac{1}{2}}.$ We prove bounds on the radical of {\it multiperfect} number $m$ depending on its abundancy index. Assuming the ABC conjecture, we apply this result to study gaps between multiperfect numbers, multiperfect numbers represented by polynomials. Finally, we prove that there are only finitely many multiperfect multirepdigit numbers in any base $g$ where the number of digits in the repdigit is a power of $2.$ This generalizes previous works of several authors including O. Klurman, F. Luca, P. Polack, C. Pomerance and others.

\end{abstract}

\tableofcontents

\section{Introduction}

A positive integer $n$ is perfect if $\sigma\left(n\right) = 2n$, where $\sigma\left(n\right)$ is the sum of all the positive integer divisors of $n$. Perfect numbers have been studied for many years since Euclid proved a formation rule whereby ${2}^{p-1}\left(2^{p}-1\right)$ is an even perfect number whenever $2^{p -1}$ is prime. He also showed that if an odd perfect number were to exist, it would be in the form $p^{\alpha}d^2,$ where $p \equiv \alpha \equiv 1 \pmod 4.$ In the study of perfect numbers, there are still two famous questions that remain wide open. First, whether there are infinitely many  perfect numbers, and second, whether there exists an odd perfect number. Though not many positive results have come with these two problems, several authors have shown that there are finitely many perfect numbers with certain properties. For example, Pollack proved this for numbers with all identical digits in \cite{lucapollack}, and Luca proved this for Fibonacci numbers in \cite{luca}. More up-to-date results and notions related to perfect numbers are introduced in \cite{Sandor}.
    $\\ \\$ Luca and Pomerance related the ABC-conjecture to the study of perfect numbers in \cite{lucap}. They first proved that $\rad\left(n\right) < 2n^{\frac{17}{26}}$. This allowed them to apply the ABC-conjecture to make conclusions about the gaps between perfect numbers. Acquaah and Konyagin gave a simple proof of a slightly weaker inequality: $\rad\left(n\right) < 2n^{\frac{2}{3}}$ in \cite{acquaahkonyagin}. Since then, Klurman proved a stronger bound in \cite{klurman}, namely $\rad\left(n\right) < 2n^{\frac{9}{14}}$. Though not proven yet, it is conjectured that $\rad\left(n\right) < 2n^{\frac{1}{2}}$. In fact, Ellia in \cite{Ellia} along with Ochem and Rao in \cite{Ochem} showed that if $\rad\left(n\right) < n^{\frac{1}{2}}$, then the special prime $p$ with odd exponent must be greater $222$, and greater than than $148207$ if $3 \nmid n$.
$\\ \\$ 
    A positive integer $m$ is multiperfect (or multiply perfect) if $\sigma\left(m\right)=km$, where $k > 2$ is referred to as the abundancy of $m$, and $m$ is called $k$-perfect. A method for determining up to 1,000,000,000 multiperfect numbers was first introduced by Carmichael \cite{RD} in 1907. Website \cite{website} provides the latest database for all known multiperfect numbers.
$\\ \\$ 
    In this paper, we prove a bound on the radical of multiperfect number through the following two theorems:
    
    \textbf{Theorem 1.} Suppose $m$ is an odd multiperfect number such that $\sigma\left(m\right) = km.$ Then, if $k$ is odd, we have that $\rad\left(m\right) \leq \sqrt{m}.$ If $k \equiv 2 \pmod 4,$ we have $\rad\left(m\right) < m^{\frac{9}{14}}.$ Otherwise, suppose $k = 2^nt$, where $t$ is odd and $n \geq 2.$ We have \[\rad\left(m\right) < m^\frac{4n + 1}{4n + 4}.\]
    
    \textbf{Theorem 2.} Suppose $m$ is an even multiperfect number such that $\sigma\left(m\right) = km.$ Then if $\nu_2\left(m\right) = n$ and $\nu_2\left(k\right) = \alpha,$ we have \[\rad m < m^\frac{2n + 2\alpha + 1}{2n + 2 \alpha + 2}.\]

 We then use this radical bound to show that the ABC-conjecture implies that generic polynomials must have finitely many $k$-perfect numbers in their range of outputs. This generalizes the work of Klurman~\cite{klurman}, who established similar result for perfect numbers (the case $k=1$). Finally, we prove an extension of Pollack's result with numbers having identical digits in \cite{Pollack}, to multiperfect numbers with multirepdigits. 

\section*{Acknowledgments}

This research was conducted at the 2018 AwesomeMath Summer Program. We would like to thank Dr. Oleksiy Klurman for being our research advisor and Mr. George Catalin Turcas for providing helpful feedback on our work.
    

\section{Preliminary Lemmas}

\begin{lemma}\label{lem1}
A number of the form $p_1 \cdot p_2 \cdot p_3^2 \cdot p_4^2$ for odd primes $p_1, p_2, p_3, p_4$ is not a $k$-perfect number if $4 \mid k.$
\end{lemma} 
\begin{proof} Since $k$ is even, let $k = 2^n \cdot t$, where $t$ is odd. Then we have the equation $2^{n} \cdot t \cdot p_{1} \cdot p_{2} \cdot {p_{3}}^{2} \cdot {p_{4}}^2 = \left(1 + p_1\right)\left(1 + p_2\right)\left(1 + p_3 + p_3^2\right)\left(1 + p_4 + p_4^2\right).$ If we divide both sides by $p_1p_2p_3^2p_4^2,$ we get:

$$2^nt = \left(1 + \frac{1}{p_1}\right)\left(1 + \frac{1}{p_2}\right)\left(1 + \frac{1}{p_3} + \frac{1}{p_3^2}\right)\left(1 + \frac{1}{p_4} + \frac{1}{p_4^2}\right) < \frac{4}{3} \cdot \frac{6}{5} \cdot \frac{31}{25} \cdot \frac{57}{49} < 4.$$ However, since $4|k$, we have that $n \geq 2,$ so this is impossible. 
\end{proof}

\begin{lemma}\label{2}

A number of the form $p_1 \cdot p_2 \cdot p_3^2 \cdot p_4^2 \cdot p_5^2$ is not a $k$-perfect number for any $k$ divisible by $4.$

\end{lemma}
\begin{proof} Since $k$ is even, let $k = 2^n \cdot t$, where $t$ is odd. Then we have the equation \[2^{n} \cdot t \cdot p_{1} \cdot p_{2} \cdot {p_{3}}^{2} \cdot {p_{4}}^2 \cdot p_5^2 = \left(1 + p_1\right)\left(1 + p_2\right)\left(1 + p_3 + p_3^2\right)\left(1 + p_4 + p_4^2\right)\left(1 + p_5 + p_5^2\right).\] If we divide both sides by $p_1p_2p_3^2p_4^2p_5^2,$ we get:

$$2^nt = \left(1 + \frac{1}{p_1}\right)\left(1 + \frac{1}{p_2}\right)\left(1 + \frac{1}{p_3} + \frac{1}{p_3^2}\right)\left(1 + \frac{1}{p_4} + \frac{1}{p_4^2}\right)\left(1 + \frac{1}{p_5} + \frac{1}{p_5^2}\right) $$ $$< \frac{4}{3} \cdot \frac{6}{5} \cdot \frac{31}{25} \cdot \frac{57}{49} \cdot \frac{133}{121} < 4.$$ However, since $4|k$, we have that $n \geq 2,$ so this is impossible. \end{proof}

\begin{lemma}\label{3}
For a prime number $p$ and odd number $e$, we have $\nu_2\left(\sigma\left(p^e\right)\right) = \nu_2\left(e + 1\right) + \nu_2\left(p + 1\right) - 1$.
\end{lemma}
\begin{proof}
Let us consider two separate cases: when $e$ is even and when it is odd. If $e$ is even, 
\[
\sigma\left(p^e\right) = 1 + p + p^2 + \cdots + p^{e-1} + p^{e} \equiv 1 + 1 + \cdots + 1 \equiv 1 \mod{2}, 
\]
This clearly contradicts the lemma. Our next case is when $e$ is odd. We may express $e + 1 = 2^r \cdot s$, for some $r$ and odd $s$. Thus $\sigma\left(p^e\right)$ can be expressed as below:
\[
\sigma\left(p^e\right) = \frac{p^{e+1} - 1}{p - 1} = \frac{p^s - 1}{p-1}\prod_{1 \leq i \leq {r-1}}{p^{s \cdot 2^i}- 1}.
\]
For every $p^{2^j * s} + 1$ term, where $i \geq 1$, $p^{2^j \cdot s} \equiv_4 1$. Therefore, $2 \parallel p^{2^j \cdot s} + 1$. We now know that since $s$ is odd, we have
\[
\frac{p^s - 1}{p - 1} = p^{s-1} + p^{s - 2} + \cdots + p + 1 \equiv_2 1 + 1 + \cdots + 1 \equiv_2 1. 
\]
If we let 
\[
p = \left(2^k \cdot t\right)^s + 1 = \left(2^j \cdot t\right)^s - \left(2^j \cdot t\right)^{s - 1} \left(2^j \cdot t\right)^{s - 2} + \cdots + \left(2^j \cdot t\right)^1
\]
Thus, we have that $j = \nu_2\left(p + 1\right)$, and $2 \parallel p^s + 1$. Finally, we conclude that $\nu_2\left(\sigma\left(p^e\right)\right) = \nu_2\left(p + 1\right) + \nu_2\left(e + 1\right) - 1$. \end{proof}

\begin{lemma}\label{4} Suppose $m$ is a $k$-perfect number, where $k$ is even and $m$ is odd. Then if $r$ is the number of distinct prime factors of $m$, we have that $\nu_2\left(k\right) < \frac{1}{3}r$ if $r \geq 4.$

\end{lemma}
\begin{proof}
Letting $k = 2^nt$ where $t$ is odd, we have the equation: $$2^ntp_1^{\alpha_1} \ldots p_r^{\alpha_r} = \left(1 + p_1 + \ldots + p_1^{\alpha_1}\right)\left(1 + p_2 + \ldots + p_2^{\alpha_2}\right)\ldots \left(1 + p_r + \ldots + p_r^{\alpha_r}\right).$$ We divide both sides by $p_1^{\alpha_1} \ldots p_r^{\alpha_r}$ to get:

$$2^nt = \left(1 + \frac{1}{p_1} + \ldots + \frac{1}{p_1^{\alpha_1}}\right)\left(1 + \frac{1}{p_2} + \ldots + \frac{1}{p_2^{\alpha_2}}\right) \ldots \left(1 + \frac{1}{p_r} + \ldots + \frac{1}{p_r^{\alpha_r}}\right).$$ Suppose that each of these series were infinite. We have:

\[ \left(1 + \frac{1}{p_1} + \ldots + \frac{1}{p_1^{\alpha_1}}\right)\left(1 + \frac{1}{p_2} + \ldots + \frac{1}{p_2^{\alpha_2}}\right) \ldots \left(1 + \frac{1}{p_r} + \ldots + \frac{1}{p_r^{\alpha_r}}\right) \] \[ < \left(1 + \frac{1}{p_1} + \ldots \right)\left(1 + \frac{1}{p_2} + \ldots \right) \ldots \left(1 + \frac{1}{p_r} + \ldots \right)\] \[ \leq \frac{3}{2} \cdot \frac{5}{4} \cdot \frac{7}{6} \cdot \frac{11}{10} \cdot \frac{5}{4} \cdot \ldots \cdot \frac{5}{4} \leq \left(\frac{5}{4}\right)^{r - 3} \cdot \frac{3}{2} \cdot \frac{7}{6} \cdot \frac{11}{10} < \left(\frac{5}{4}\right)^r. \] From this, we get $2^n \leq 2^nt < \left(\frac{5}{4}\right)^r.$ Thus, we have that $n < r \log_2{1.25} < \frac{1}{3}r$ as desired. \end{proof}

\begin{lemma}\label{LOOPY}

Suppose $A = k_{e}k_{e - 1} \ldots k_1$ and $B = k_ek_{e - 1} \ldots k_2 + k_ek_{e - 1} \ldots k_3 + \ldots + k_ek_{e - 1} + k_e$ where $k_i \geq 2.$ Then if the integer $p_y \geq 3 \cdot 2^e,$ it is impossible to have $Ap_y - B \mid p_y^2 + p_y + 1.$ 

\end{lemma} 
\begin{proof} Suppose $Ap_y - B \mid p_y^2 + p_y + 1.$ Then $Ap_y - B \mid Ap_y^2 + Ap_y + A - \left(Ap_y^2 - Bp_y\right) = Ap_y + Bp_y + A.$ From this, we have $Ap_y - B \mid Bp_y + A + B.$ Then $Ap_y - B \leq Bp_y + A + B,$ so $p_y \leq \frac{A + 2B}{A - B}.$ Then if $p_y \geq c,$ we have $Ac - Bc \leq A + 2B,$ so $A\left(c - 1\right) \leq B\left(c + 2\right).$ Thus $\frac{c - 1}{c + 2} \leq \frac{B}{A}.$ Plugging in $A$ and $B$ gives us:

$$\frac{c - 1}{c + 2} \leq \frac{1}{k_1} + \frac{1}{k_1k_2} + \ldots + \frac{1}{k_1k_2 \ldots k_e} \leq \frac{1}{2} + \frac{1}{4} + \ldots + \frac{1}{2^e} \leq \frac{2^e - 1}{2^e}.$$ Then $2^ec - 2^e \leq 2^ec - c + 2 \cdot 2^e - 2,$ which is false for all $c \geq 3 \cdot 2^e - 2.$ Since $p_y \geq 3 \cdot 2^e - 2,$ this is impossible as desired. \end{proof}

\section{Bounds on the radical} 

\subsection{Odd Multiperfect Numbers}

\begin{theorem}\label{oddmain} Suppose $m$ is an odd multiperfect number such that $\sigma\left(m\right) = km.$ Then, if $k$ is odd, we have that $\rad\left(m\right) \leq \sqrt{m}.$ If $k \equiv 2 \pmod 4,$ we have $\rad\left(m\right) < m^{\frac{9}{14}}.$ Otherwise, suppose $k = 2^nt$, where $t$ is odd and $n \geq 2.$ We have \[\rad\left(m\right) < m^\frac{4n + 1}{4n + 4}.\]

\end{theorem}
\begin{proof} We begin by considering if $k, m$ are both odd in the equation $\sigma\left(m\right) = km.$ Let $m = p_1^{\alpha_1} \ldots p_r^{\alpha_r}.$ We have:

$$km = \left(1 + p_1 + \ldots + p_1^{\alpha_1}\right)\ldots \left(1 + p_r + \ldots + p_r^{\alpha_r}\right).$$ Since $k, m$ are both odd, so is $km.$ This means that the product on the right in the equation above is also odd, so each sum of the powers of the $p_i$ must be odd for all $i$ where $1 \leq i \leq r$. This implies that all of the nonzero $\alpha_i$ are even, so they are at least $2.$ Therefore $\rad\left(m\right) \leq \sqrt{m}$ if $k, m$ are both odd. Now we consider what happens when $k$ is even.

If $m$ is squarefull, then its radical is clearly no larger than $\sqrt{m}.$ Therefore, we may assume that some of the prime factors of $m$ are raised to only the first power. Let $$m = p_1p_2 \ldots p_gp_{g + 1}^{2\alpha_{g+1} + 1} \ldots p_s^{2\alpha_s + 1}p_{s + 1}^{2\alpha_{s + 1}} \ldots p_r^{2\alpha_r}$$
and $$\beta = \frac{4n + 1}{4n + 4}.$$ Suppose by contradiction that $\rad\left(m\right) > m^{\beta}.$ This would give us:
\[
p_1p_2 \ldots p_r > p_1^{\beta}p_2^{\beta} \ldots p_g^{\beta}p_{g + 1}^{2\alpha_{g + 1}\beta + \beta} \ldots p_s^{2\alpha_s\beta + \beta}p_{s + 1}^{2\alpha_{s + 1}\beta} \ldots p_r^{2\alpha_r\beta}.
\] From this, we get:

\[
p_1p_2 \ldots p_g > p_{g + 1}^{\frac{2\alpha_{g + 1}\beta + \beta - 1}{1 - \beta}} \ldots p_s^{\frac{2\alpha_s\beta + \beta - 1}{1 - \beta}}p_{s + 1}^{\frac{2\alpha_{s + 1}\beta - 1}{1 - \beta}} \ldots p_r^{\frac{2\alpha_r\beta - 1}{1 - \beta}}. 
\] Since $\sigma\left(m\right) = km = 2^ntm,$ we have:

\[
2^ntp_1p_2 \ldots p_gp_{g + 1}^{2\alpha_{g+1} + 1} \ldots p_s^{2\alpha_s + 1}p_{s + 1}^{2\alpha_{s + 1}} \ldots p_r^{2\alpha_r} 
\]
$$
= \left(1 + p_1\right) \ldots \left(1 + p_g\right)\left(1 + p_{g + 1} + \ldots + p_{g + 1}^{2\alpha_{g + 1}+1}\right) \ldots \left(1 + p_s + \ldots + p_s^{2\alpha_s + 1}\right)$$$$\left(1 + p_{s + 1} + \ldots + p_{s+1}^{2\alpha_{s + 1}}\right) \ldots \left(1 + p_r + \ldots + p_r^{2\alpha_r}\right).
$$ Without loss of generality, let $p_1 < p_2 < \ldots <p_g.$ This means $p_g$ cannot divide $1 + p_1, 1 + p_2, \ldots, 1 + p_g.$ We have 2 cases: either $p_g|1 + p_{g + 1} + \ldots + p_{g + 1}^{2\alpha_{g + 1} + 1}$ or $p_g|1 + p_{s + 1} + \ldots + p_{s+1}^{2\alpha_{s + 1}}$ where $p_{g + 1}$ and $p_{s + 1}$ are arbitrary primes with odd and even exponents, respectively.
$\\$
$\\$
\textit{Case 1}: $p_g|1 + p_{g + 1} + \ldots + p_{g + 1}^{2\alpha_{g + 1} + 1}.$
$\\$ We have $p_g < p_{g + 1}^{2\alpha_{g + 1} + 1}\left(1 + \epsilon\right).$ Combining this with our above inequality, we have: $$p_{g + 1}^{\left(2\alpha_{g + 1} + 1\right)g} > p_g^g >  p_1p_2 \ldots p_g $$$$ > p_{g + 1}^{\frac{2\alpha_{g + 1}\beta + \beta - 1}{1 - \beta}} \ldots p_s^{\frac{2\alpha_s\beta + \beta - 1}{1 - \beta}}p_{s + 1}^{\frac{2\alpha_{s + 1}\beta - 1}{1 - \beta}} \ldots p_r^{\frac{2\alpha_r\beta - 1}{1 - \beta}} > p_{g + 1}^{\frac{2\alpha_{g + 1}\beta + \beta - 1}{1 - \beta}}.$$ From $n \geq s \geq g,$ we get that $2n\alpha_{g + 1} + n > \frac{2\alpha_{g + 1}\beta + \beta - 1}{1 - \beta}.$ From $\beta = \frac{4n + 1}{4n + 4}$ we get:

$$2n\alpha_{g + 1} + n > \frac{\left(8n + 2\right)\alpha_{g + 1} - 3}{3} \implies \alpha_{g + 1} < 1.5 \implies \alpha_{g + 1} = 1.$$ Then, $$p_{g + 1}^{\frac{n + 1}{3}} > p_{g + 2}^{\frac{\left(8n + 2\right)\alpha_{g + 2} - 3}{3}} \ldots p_s^{\frac{\left(8n + 2\right)\alpha_{s} - 3}{3}}p_{s + 1}^{\frac{\left(8n + 2\right)\alpha_{s + 1} - \left(4n + 4\right)}{3}} \ldots p_r^{\frac{\left(8n + 2\right)\alpha_{r} - \left(4n + 4\right)}{3}}, $$or $$p_{g + 1} > p_{g + 2}^{\frac{\left(8n + 2\right)\alpha_{g + 2} - 3}{n + 1}} \ldots p_s^{\frac{\left(8n + 2\right)\alpha_{s} - 3}{n + 1}}p_{s + 1}^{\frac{\left(8n + 2\right)\alpha_{s + 1} - \left(4n + 4\right)}{n + 1}} \ldots p_r^{\frac{\left(8n + 2\right)\alpha_{r} - \left(4n + 4\right)}{n + 1}}.$$ Now consider what $p_{g + 1}$ divides. Suppose $p_{g + 1} \mid 1 + p_{g + 2} + \ldots + p_{g + 2}^{2\alpha_{g + 2} + 1}.$ Then $2\alpha_{g + 2} + 1 > \frac{\left(8n + 2\right)\alpha_{g + 2} - 3}{n + 1},$ or $6n\alpha_{g + 2} < n + 4.$ Contradiction. Suppose $p_{g + 1} \mid 1 + p_{s + 1} + \ldots + p_{s + 1}^{2\alpha_{s + 1}}.$ Then $2\alpha_{s + 1} \geq \frac{\left(8n + 2\right)\alpha_{s + 1} - \left(4n + 4\right)}{n + 1}.$ This means $6n\alpha_{s + 1} \leq 4n + 4,$ which we know is false for $n \geq 3.$ If $n = 2,$ that forces $\alpha_{s + 1} = 1$ and means $p_{s+ 1}$ is the only prime on the right. Then $m$ has exactly $g + 2$ prime factors, but by Lemma \ref{LOOPY}:

$$g \leq n < \frac{1}{3}\left(g + 2\right)$$ which is false for all $g \geq 1.$ Therefore we have a contradiction for all $n \geq 2.$The only other possibility is that $p_{g + 1} \mid 1 + p_i$ where $1 \leq i \leq g.$ Then $p_i \geq 2p_{g + 1} - 1.$ Since $p_{g + 1}$ was too large to divide the series on the right and $p_i > p_{g + 1},$ the same is true for $p_i$ so we must have $p_i \mid 1 + p_j$ and so on. Note $$p_{g + 1} > p_{g + 2}^{\frac{\left(8n + 2\right)\alpha_{g + 2} - 3}{n + 1}} \ldots p_s^{\frac{\left(8n + 2\right)\alpha_{s} - 3}{n + 1}}p_{s + 1}^{\frac{\left(8n + 2\right)\alpha_{s + 1} - \left(4n + 4\right)}{n + 1}} \ldots p_r^{\frac{\left(8n + 2\right)\alpha_{r} - \left(4n + 4\right)}{n + 1}}$$$$ > 3^{r - \left(g + 1\right)} > 2 \cdot 3^{2n} > 3 \cdot 2^n.$$ Writing $p_j = k_1p_i - 1$ and so on, we use Lemma \ref{LOOPY} to eliminate this case.
$\\$ $\\$
\textit{Case 2}: $p_g|1 + p_{s + 1} + \ldots + p_{s+1}^{2\alpha_{s + 1}}$ $\\$
From $p_g|1 + p_{s + 1} + \ldots + p_{s+1}^{2\alpha_{s + 1}},$ we have that $$p_g < p_{s + 1}^{2\alpha_{s + 1}} + 2p_{s + 1}^{2\alpha_{s + 1} - 1} \leq  p_{s + 1}^{2\alpha_{s + 1}} \cdot \left(1 + \epsilon\right).$$ Combining this with our inequality above, we have:

$$p_{s + 1}^{2\alpha_{s + 1}g} > p_g^g > p_1p_2 \ldots p_g $$ $$> p_{g + 1}^{\frac{2\alpha_{g + 1}\beta + \beta - 1}{1 - \beta}} \ldots p_s^{\frac{2\alpha_s\beta + \beta - 1}{1 - \beta}}p_{s + 1}^{\frac{2\alpha_{s + 1}\beta - 1}{1 - \beta}} \ldots p_r^{\frac{2\alpha_r\beta - 1}{1 - \beta}} > p_{s + 1}^{\frac{2\alpha_{s + 1}\beta - 1}{1 - \beta}}.$$ From this we get that $2\alpha_{s + 1}g > \frac{2\alpha_{s+1}\beta - 1}{1 - \beta}.$ From $n \geq g$ and $\beta = \frac{4n + 1}{4n + 4}$ we have:

$$2n\alpha_{s + 1} \geq \frac{2\alpha_{s + 1}\beta - 1}{1 - \beta} = \frac{\left(8n + 2\right)\alpha_{s + 1}  - \left(4n + 4\right)}{3}.$$ This implies that $\alpha_{s + 1} \leq 2,$ but if $\alpha_{s + 1} = 2,$ then $p_{s + 1}$ is the only prime on the right in the above inequality, violating Lemma \ref{LOOPY}. Therefore, assume $\alpha_{s + 1} = 1.$ Plugging this into the inequality above, we get:

$$2p_{s + 1}^{2n} \geq 2p_{s + 1}^{2g} > p_g^g > p_1p_2 \ldots p_g $$ $$> p_{g + 1}^{\frac{2\alpha_{g + 1}\beta + \beta - 1}{1 - \beta}} \ldots p_s^{\frac{2\alpha_s\beta + \beta - 1}{1 - \beta}}p_{s + 1}^{\frac{2\beta - 1}{1 - \beta}} \ldots p_r^{\frac{2\alpha_r\beta - 1}{1 - \beta}} > p_{s + 1}^{\frac{2\alpha_{s + 1}\beta - 1}{1 - \beta}}.$$ From $\beta = \frac{4n + 1}{4n + 4},$ we have that $\frac{2\beta - 1}{1 - \beta} = \frac{4n - 2}{3}.$ We have:

$$p_{s + 1}^{\frac{2n + 2}{3}} \geq p_{g + 1}^{\frac{\left(8n + 2\right)\alpha_{g + 1} - 3}{3}} \ldots p_s^{\frac{\left(8n + 2\right)\alpha_s - 3}{3}}p_{s + 2}^{\frac{\left(8n + 2\right)\alpha_{s + 2} - \left(4n + 4\right)}{3}} \ldots p_r^{\frac{\left(8n + 2\right)\alpha_r - \left(4n + 4\right)}{3}}.$$ Raising both sides to the $\frac{3}{2n + 2}$ power gives us:

$$p_{s + 1} \geq p_{g + 1}^{\frac{\left(8n + 2\right)\alpha_{g + 1} - 3}{2n + 2}} \ldots p_s^{\frac{\left(8n + 2\right)\alpha_s - 3}{2n + 2}}p_{s + 2}^{\frac{\left(8n + 2\right)\alpha_{s + 2} - \left(4n + 4\right)}{2n + 2}} \ldots p_r^{\frac{\left(8n + 2\right)\alpha_r - \left(4n + 4\right)}{2n + 2}}.$$ Now we consider what $p_{s + 1}$ divides. Suppose $p_{s + 1}^2 \mid 1 + p_{g + 1} + \ldots + p_{g + 1}^{2\alpha_{g + 1} + 1}.$ Then $$p_{g + 1}^{\alpha_{g + 1} + 1}\left(\sqrt{2}\right) > p_{s + 1} $$ $$\geq p_{g + 1}^{\frac{\left(8n + 2\right)\alpha_{g + 1} - 3}{2n + 2}} \ldots p_s^{\frac{\left(8n + 2\right)\alpha_s - 3}{2n + 2}}p_{s + 2}^{\frac{\left(8n + 2\right)\alpha_{s + 2} - \left(4n + 4\right)}{2n + 2}} \ldots p_r^{\frac{\left(8n + 2\right)\alpha_r - \left(4n + 4\right)}{2n + 2}} $$ $$> p_{g + 1}^{\frac{\left(8n + 2\right)\alpha_{g + 1} - 3}{2n + 2}}.$$ Then $\alpha_{g + 1} + 1 \geq \frac{\left(8n + 2\right)\alpha_{g + 1} - 3}{2n + 2}.$ This is impossible for $n \geq 2$, contradiction. Now suppose $p_{s + 1} \mid 1 + p_{g + 1} + \ldots + p_{g + 1}^{2\alpha_{g + 1} + 1}$ and $p_{s + 1} \mid 1 + p_{g + 2} + \ldots + p_{g + 2}^{2\alpha_{g + 2} + 1}.$ Then $$2p_{g + 1}^{\alpha_{g + 1} + 1}p_{g + 2}^{\alpha_{g + 2} + 1} > p_{s + 1} \geq p_{ g + 1}^{\frac{\left(8n + 2\right)\alpha_{g + 1} - 3}{2n + 2}} p_{g + 2}^{\frac{\left(8n + 2\right)\alpha_{g + 2} - 3}{2n + 2}}.$$ Contradiction. Now suppose $p_{s + 1}^2 \mid 1 + p_{s + 2} + \ldots + p_{s + 2}^{2\alpha_{s + 2}}.$ Then: $$p_{s + 2}^{\alpha_{s + 2}}\left(\sqrt{2}\right) > p_{s + 1} $$ $$\geq p_{g + 1}^{\frac{\left(8n + 2\right)\alpha_{g + 1} - 3}{2n + 2}} \ldots p_s^{\frac{\left(8n + 2\right)\alpha_s - 3}{2n + 2}}p_{s + 1}^{\frac{\left(8n + 2\right)\alpha_{s + 1} - \left(4n + 4\right)}{2n + 2}} \ldots p_r^{\frac{\left(8n + 2\right)\alpha_r - \left(4n + 4\right)}{2n + 2}}$$ $$ > p_{s + 2}^{\frac{\left(8n + 2\right)\alpha_{s + 2} - \left(4n + 4\right)}{2n + 2}}.$$ Contradiction. Now suppose $p_{s + 1} \mid 1 + p_{s + 2} + \ldots + p_{s + 2}^{2\alpha_{s + 2}}$ and $p_{s + 1} \mid 1 + p_{s + 3} + \ldots + p_{s + 3}^{2\alpha_{s+ 3}}.$ Then: $$p_{s + 2}^{\alpha_{s + 2}}p_{s + 3}^{\alpha_{s + 3}} > p_{s + 1} \geq p_{s + 2}^{\frac{\left(8n + 2\right)\alpha_{s + 2} - 3}{2n + 2}} p_{s + 3}^{\frac{\left(8n + 2\right)\alpha_{s + 3} - 3}{2n + 2}}.$$ Contradiction. Now suppose $p_{s + 1} \mid 1 + p_{g + 1} + \ldots + p_{g + 1}^{2\alpha_{g + 1} + 1}$ and $1 + p_{s + 2} + \ldots + p_{s + 2}^{2\alpha_{s + 2}}.$ Then: $$p_{g + 1}^{\alpha_{g + 1} + 1}p_{s + 2}^{\alpha_{s + 2}} > p_{s + 1} > p_{g + 1}^{\frac{\left(8n + 2\right)\alpha_{g + 1} - 3}{2n + 2}} p_{s + 2}^{\frac{\left(8n + 2\right)\alpha_{s + 2} - \left(4n + 4\right)}{2n + 2}}.$$ Contradiction. Now suppose that $p_{s + 1} \mid 1 + p_i$ and $p_{s + 1} \mid 1 + p_{g+1} + \ldots + p_{g + 1}^{2\alpha_{g + 1} + 1}.$ Then: $$2p_{g + 1}^{2\alpha_{g + 1} + 1} > p_{s + 1} \geq p_{g + 1}^{\frac{\left(8n + 2\right)\alpha_{g + 1} - 3}{2n + 2}}.$$ From this we have that $2\alpha_{g + 1} + 1 \geq \frac{\left(8n + 2\right)\alpha_{g + 1} - 3}{2n + 2},$ or $\alpha_{g + 1} \leq \frac{2n + 5}{4n - 2}.$ For $n \geq 4$ this is clearly impossible. Thus we only need to consider if $n = 2$ or $n = 3.$
$\\$ $\\$
Case 2.1: $n = 2$ $\\$ 
The inequality above implies $\alpha_{g + 1} = 1.$ We recall that $1 \leq g \leq s \leq n = 2.$ If $g = s = 2,$ then $p_3 \mid 1 + p_3 + \ldots + p_3^{2\alpha_3 + 1},$ which is absurd. Thus we would have to have $g = 1$ and $s = 2.$

However, this means that $p_3 \mid 1 + p_1,$ $p_3 \mid 1 + p_2 + p_2^2 + p_2^3$ and $p_1 \mid 1 + p_3 + p_3^2.$ This is impossible by Lemma \ref{LOOPY}.
$\\$ $\\$
Case 2.2: $n = 3$ $\\$
As before, we can't have $g = s.$ Thus we have $\left(g, s\right) = \left(1, 2\right), \left(1, 3\right)$ or $\left(2, 3\right).$ Since $p_{g + 1}$ is odd, we note that we have $4 \mid 1 + p_{g + 1} + p_{g + 1}^2 + p_{g + 1}^3.$ This means $s < n,$ so the only case to consider is $g = 1$ and $s = 2,$ which is eliminated analogously to case 2.1. Now suppose that $p_{s + 1} \mid 1 + p_i$ and $p_{s + 1} \mid 1 + p_{s + 2} + \ldots + p_{s + 2}^{2\alpha_{s + 2}}.$ From $p_{s + 1} \mid 1 + p_{s + 2} + \ldots + p_{s + 2}^{2\alpha_{s + 2}},$ we get that:

$$2p_{s + 2}^{2\alpha_{s + 2}} > p_{s + 1} > p_{s + 2}^{\frac{\left(8n + 2\right)\alpha_{s + 2} - \left(4n + 4\right)}{2n + 2}}.$$ This forces $\alpha_{s + 2} = 1.$ Now suppose $p_{s + 2} \mid 1 + p_{s + 3} + \ldots + p_{s + 3}^{2\alpha_{s + 3}}.$ Then, using the implied inequalities, we get that $\alpha_{s + 3} \leq \frac{2n + 2}{4n - 5.}$ This is impossible for all $n \geq 4,$ and $n = 2$ and $n = 3$ are eliminated as above. Therefore, we must have that $p_{s + 3} \mid 1 + p_d$ and $1 + p_e$ where $1 \leq d, e \leq g.$ Then we have $p_d \mid 1 + p_f$ and so on, which is eliminated by Lemma \ref{LOOPY}. \end{proof}

\subsection{Even Multiperfect Numbers}

\begin{theorem}\label{evenmain}

Suppose $m$ is an even multiperfect number such that $\sigma\left(m\right) = km$. Let $k = 2^nt$, where $t$ is odd and $n\geq0$. Suppose $m=2^{\alpha}h$, where $h$ is odd and $\alpha \geq1.$  We have \[\rad\left(m\right) < m ^{\frac{2n+2\alpha+1}{2n+2\alpha+2}}.\]
\end{theorem} \begin{proof} 
Let $m=2^{\alpha}p_1p_2 \ldots p_i p_{i+1}^{2\alpha_{i+1}+1}\ldots p_j^{2\alpha_j+1}p_{j+1}^{2\alpha_{j+1}}\ldots p_k^{2\alpha_k}$. Assume that $\rad\left(m\right) > m^{\beta}$, where $\beta = \frac{2n+2\alpha+1}{2n+2\alpha+2}$. Without loss of generality, suppose that $p_1 <p_2<\ldots<p_i$. Suppose that for large $m$, we have $$\rad\left(m\right)=2p_1\ldots p_i\ldots p_j\ldots p_k > 2^{\beta}p_1^{2\alpha_1\beta}\ldots p_{i+1}^{2\alpha_{i+1}\beta+\beta}\ldots p_{j+1}^{2\alpha_{j+1}\beta}\ldots p_k^{2\alpha_k\beta}$$ Since $m$ is multiperfect, we have \[2^{n+\alpha} t p_1\dots p_i{p_{i+1}}^{2\alpha_{i+1}+1}\dots{p_j}^{2\alpha_j+1}\dots {p_k}^{2\alpha_k} \] \[ = \left(1+2+\dots+2^{\alpha}\right)\prod_{x=1}^{i}\left(1+p_x\right)\prod_{y=i+1}^{j}\left(1+p_y+\dots+{p_y}^{2\alpha_y+1}\right)\prod_{z=j+1}^{k}\left(1+p_z+\dots+{p_z}^{2\alpha_z}\right)\]
$\\$
\textit{Case 1}: If $\alpha = 1$, we have    
\[ 2p_1p_2\dots p_i \geq \prod_{x=i+1}^{j} p_x^{\frac{2\alpha_x\beta+\beta-1}{1-\beta}}\prod_{x=j+1}^{k} p_x^{\frac{2\alpha_x\beta-1}{1-\beta}} \] If $i = 0$, since $2^{n+1}\mid \sigma\left(m\right)$, $j\neq0$. We have $\rad\left(m\right) \gg \sqrt{m}$ and the result follows. So we can assume that $i \geq 1$. Since we assume that $\rad\left(m\right) \gg m^{\beta}$,

\[2p_1\dots p_i\geq \prod_{x=i+1}^{j}{p_x}^{\frac{2\alpha_x\beta+\beta-1}{1-\beta}}\prod_{y=j+1}^{k}{p_y}^{\frac{2\alpha_y\beta-1}{1-\beta}}\]Since 

\[p_i\mid \sigma\left(m\right),\]\[ p_i\nmid \prod_{x=1}^{i}\left(1+p_x\right), \]
so $p_i$ must divide one of\[ 1+p_c+\cdots+p_c^{2\alpha_c}\left(c\in[i+1,j]\right) \]or\[ 1+p_d+\cdots+p_d^{2\alpha_d+1} \left(d\in[j+1,k]\right)\]If $p_i\mid 1+p_c+\cdots+p_c^{2\alpha_c}$, combining that 
 \[\frac{\sigma\left(p^{\alpha}\right)}{p^{\alpha}}=1+\frac{1}{p}+\dots+\frac{1}{p^{\alpha}}<\frac{1-{\left(\frac{1}{p}\right)}^{n}}{1-\frac{1}{p}}=\frac{p}{p-1}<\frac{3}{2}\] then $p_i <\sigma\left({p_c}^{2\alpha_c}\right)< \frac{3}{2}{p_c}^{2\alpha_c}$. From Lemmas \ref{2} and \ref{3}, we have that $m$ must have at least 3 distinct prime divisors. From Lemma \ref{4}, we have $i\leq v_2\left(m\right)=n+\alpha$. Thus, we have \[{\left(\frac{3}{2}\right)}^{n+1}2{p_c}^{2\alpha_c\left(n+1\right)}\geq2\left(\frac{3}{2}\right)^{i}{p_c}^{2\alpha_c i}>2p_1\dots p_i \] \[ \geq  \prod_{x=i+1}^{j}{p_x}^{\frac{2\alpha_x\beta+\beta-1}{1-\beta}}\prod_{y=j+1}^{k}{p_y}^{\frac{2\alpha_y\beta-1}{1-\beta}} \] We proceed by comparing exponents of primes. For $\alpha_c=1, 2\left(n+1\right)=\frac{2\beta+\beta-1}{1-\beta}$, when $\beta=\frac{2n+2\alpha+1}{2n+2\alpha+3}$.$\\$ For $\alpha_c \geq 2$, $$2\alpha_c\left(n+1\right)>\frac{2\alpha_c\beta+\beta-1}{1-\beta}$$because $\frac{d}{d\alpha_c}\left(2\alpha_c\left(n+1\right)\right)=2n+2>\frac{d}{d\alpha_c}\left(\frac{2\alpha_c\beta+\beta-1}{1-\beta}\right)=n+\frac{3}{2}$. We have ${p_c}^{2\alpha\left(n+1\right)} \leq {p_c}^{\frac{2\alpha_c\beta+\beta-1}{1-\beta}}$ and $\frac{3}{2}<p_e$, for $1<e\leq k$. Thus, 
\[2\left(\frac{3}{2}\right)^{i}{p_c}^{2\alpha_c\left(n+1\right)}<\prod_{x=i+1}^{j}{p_x}^{\frac{2\alpha_x\beta+\beta-1}{1-\beta}}\prod_{y=j+1}^{k}{p_y}^{\frac{2\alpha_y\beta-1}{1-\beta}} \]
which is a contradiction. Else if $p_i \mid 1+p_d+\cdots+p_d^{2\alpha_d+1}$, then $p_i<\frac{3}{2}{p_d}^{2\alpha_d+1}$. By analogy with the previous proof ,\[2\left(\frac{3}{2}\right)^{n+1}{p_d}^{\left(2\alpha_d+1\right)\left(n+1\right)}\geq2\left(\frac{3}{2}\right)^{i}{p_d}^{\left(2\alpha_d+1\right) i}\] \[>2p_1\dots p_i\geq \prod_{x=i+1}^{j}{p_x}^{\frac{2\alpha_x\beta+\beta-1}{1-\beta}}\prod_{y=j+1}^{k}{p_y}^{\frac{2\alpha_y\beta-1}{1-\beta}} \] and $\left(2\alpha_d+1\right)\left(n+1\right)\leq \frac{\left(2\alpha_d+1\right)\beta-1}{1-\beta}$, for $\beta=\frac{2\alpha+2n+1}{2\alpha+2n+2}$, which is a contradiction.
$\\$ $\\$
\textit{Case 2}:
If $\alpha \geq 2$, we have
\[p_1p_2\dots p_i \geq 2^{\frac{\alpha\beta-1}{1-\beta}}\prod_{x=i+1}^{j} p_x^{\frac{2\alpha_x\beta+\beta-1}{1-\beta}}\prod_{x=j+1}^{k} p_x^{\frac{2\alpha_x\beta-1}{1-\beta}} \]Similar to Case 1, we can assume that $i\geq1$. If $p_i\mid 1+p_c+\cdots+p_c^{2\alpha_c}$, where $\left(c\in[i+1,j]\right)$, we have
 \[\left(\frac{3}{2}\right)^{n+\alpha}{p_c}^{2\alpha_c\left(n+\alpha\right)}\geq\left(\frac{3}{2}\right)^{i}{p_c}^{2\alpha_c i}>p_1\dots p_i \] \[\geq 2^{\frac{\alpha\beta-1}{1-\beta}} \prod_{x=i+1}^{j}{p_x}^{\frac{2\alpha_x\beta+\beta-1}{1-\beta}}\prod_{y=j+1}^{k}{p_y}^{\frac{2\alpha_y\beta-1}{1-\beta}} \]Also,  $${p_c}^{2\alpha_c\left(n+\alpha\right)} \leq {p_c}^{\frac{2\alpha_c\beta+\beta-1}{1-\beta}}$$ for $\beta=\frac{2\alpha+2n+1}{2\alpha+2n+3}$, which is impossible. Else if $p_i \mid 1+p_d+\cdots+p_d^{2\alpha_d+1}$, we have
 \[\left(\frac{3}{2}\right)^{n+\alpha}{p_d}^{\left(2\alpha_d+1\right)\left(n+\alpha\right)}\geq\left(\frac{3}{2}\right)^{i}{p_d}^{\left(2\alpha_d+1\right) i}>2p_1\dots p_i\] \[\geq \prod_{x=i+1}^{j}{p_x}^{\frac{2\alpha_x\beta+\beta-1}{1-\beta}}\prod_{y=j+1}^{k}{p_y}^{\frac{2\alpha_y\beta-1}{1-\beta}} \]Also,  ${p_d}^{\left(2\alpha_d+1\right)\left(n+\alpha\right)} \leq {p_d}^{\frac{2\alpha_d\beta+\beta-1}{1-\beta}}$ for $\beta=\frac{2\alpha+2n+1}{2\alpha+2n+2}$, which is again impossible. \end{proof}

\section{ABC Conjecture and Multiperfect Numbers} 

We proceed by applying our bounds to study the gaps between multiperfect and perfect numbers. Recall the ABC Conjecture, which states that for a fixed $\epsilon > 0$, there exists a constant $C_\epsilon$ (dependent on $\epsilon$) such that for all coprime $a , b \in \mathbb{N}_{>0}$ and $a + b = c$, the following inequality is true:
\[
\max\left(|a|, |b|, |c|\right) \leq C_{\epsilon}\rad\left(abc\right)^{1 + \epsilon}
\]
\subsection{Multiperfect numbers and polynomials}
\begin{theorem}\label{thm8}
Assume that the ABC conjecture is true. Suppose that $f\left(x, y\right) \in \mathbb{Z}[x, y]$ is homogeneous, without any repeated linear factors. Fix $\epsilon > 0$. Then, for any coprime integers $m$, $n$, $$\rad\left(f\left(m, n\right)\right) \gg \max\left(|m|, |n|\right)^{\deg{f} - 2 - \epsilon}$$.
\end{theorem}

\begin{theorem}
Assume that the ABC conjecture is true. Let $f\left(X\right) \in \mathbb{Z}[X]$ be a polynomial of degree $d \geq 1$ without repeated roots. Fix $\epsilon > 0$. Then $$\rad\left(f\left(n\right)\right) \gg |n|^{d-1-\epsilon}$$.
\end{theorem}
\begin{prop}
Assume that the ABC conjecture is true. Suppose $P\left(x\right) \in \mathbb{Z}[x]$ of degree $> \lfloor {\frac{4n+4}{3}} \rfloor$ where $n = \nu_2\left(k\right) \geq 2$ has no repeated factors. Then, there are only finitely many integers $q$, such that $P\left(q\right)$ is an odd multiperfect number.
\end{prop}

\begin{proof}
Suppose $P\left(q\right)$ is an odd multiperfect number with a large value of $q$, $n = \nu_2{k} \geq 2$, and $\deg{P} = d > \lfloor {\frac{4n+4}{3}} \rfloor$. Fix $\epsilon > 0$. By Theorems $1$, and $7$, $$q^{\frac{4n+1}{4n+4}d} \gg \rad\left(P\left(q\right)\right) \gg q^{d-1-\epsilon}$$ Combining the two, $$\frac{4n+1}{4n+4}d \geq d-1-\epsilon \implies d \leq \frac{4n+4}{3} \left(1+\epsilon\right)$$ Since it is possible to choose $\epsilon$ such that $d \leq \frac{4n+4}{3} \left(1+\epsilon\right) < \lfloor{\frac{4n+4}{3}} \rfloor + 1$, the contradiction implies the desired result.
\end{proof}

\begin{prop}
Assume that the ABC- conjecture is true. Let $f\left(x, y\right) \in \mathbb{Z}[x]$ be a homogeneous form of degree $> 2\lfloor {\frac{4n+4}{3}} \rfloor$ where $n = \nu_2\left(k\right) \geq 2$ without repeated linear factors. Then there are only finitely many perfect numbers of the form $f\left(m, n\right)$ for $m, n \in \mathbb{Z}$.
\end{prop}

\begin{proof}
By Theorems \ref{oddmain} and \ref{thm8}, we have:

$$
\max\left(|m|, |n|\right)^{\frac{4n+1}{4n+4}d} \gg \rad\left(f\left(m, n\right)\right) \gg \max\left(|m|, |n|\right)^{d-2-\epsilon}$$ $$ \implies \frac{4n+1}{4n+4}d > d - 2 - \epsilon
\implies  d > \frac{4n+4}{3}\left(2+\epsilon\right)
$$

Choosing $\epsilon$ small enough yields the desired conclusion.
\end{proof}

\subsection{Distance Between Perfect and Multiperfect Numbers} Luca and Pomerance~\cite{lucap} showed that under the assumption of the ABC conjecture, the equation 
\begin{equation}
x - y = k
\end{equation}
has only finitely many solutions in perfect numbers $x$ and $y$ when $k$ is odd. They also prove that \textit{Equation 1} has finitely many solutions when $x, y$ are perfect and squarefull \cite{luca}. We consider similar question: how close can perfect ad multiperfect number come together? 
\begin{prop}\label{prop12}
Assume the ABC Conjecture holds. Then there exists finitely many solutions to \textit{Equation 1} when $k$ is odd, $x$ is an even perfect number and $y$ is an odd multiperfect number. 
\end{prop}
\begin{prop}\label{prop13}
Assume the ABC Conjecture holds. Then there exists finitely many solutions to \textit{Equation 1} when $k$ is odd, $x$ is an even perfect number and $y$ is an even multiperfect number.
\end{prop}
The proofs of Propositions \ref{prop12} and \ref{prop13} are analogous to those of Luca and Pomerance in \cite{lucap}. The only difference is inserting the bounds on radical in Theorems \ref{oddmain} and \ref{evenmain} for even and odd multiperfect numbers. 

\subsection{Multirepdigit multiperfect numbers}

Pollack recently proved~ \cite{Pollack} that there exist finitely many \textit{"repdigit"} perfect and multiperfect numbers. Along with Luca, he also showed that there exist only finitely many repdigit multiperfect numbers in any base \textit{g} \cite{lucapollack}. We will prove the following extension:
\begin{theorem}\label{multirep}
There exist only finitely many multiperfect \textit{multi}repdigit numbers in any base \textit{g} if the abundancy of the multiperfect number is a power of $2.$ 
\end{theorem}
 While we define repdigit numbers in the following way: consider the Lucas sequence  
\[
U_n = \frac{g^n - 1}{g - 1} ,
\]
 the repdigit number in some base $g$ will be defined as
\[
D \cdot \frac{g^n - 1}{g - 1};  \quad D \in \{1, 2, \cdots g-1\}. 
\]
On the other hand, a multirepdigit number will be the case when $D \in \mathbb{N}$, rather than being restricted by its base. We will focus on the bounding of multirep digit numbers when the abundancy is $2^s$ for a positive integer $s$.
$\\$ $\\$
We know that $U_1 \mid U_2 \mid U_4 \mid U_8 \mid \cdots \mid U_{2^s} \mid \cdots$. Thus, we know that $\frac{\sigma\left(D \cdot U_{2^s}\right)}{D \cdot U_{2^s}}$ will be strictly increasing when $s = 1, 2, \cdots$. We also know that $\frac{\sigma\left(n\right)}{n} = \sum_{d \mid n} {\frac{1}{d}}$. Since $D$ is not bounded (it ranges throughout the natural numbers), we will work on restricting the value of $l$ rather than $D$, where $l$ is the abundancy of the multiperfect number. We know that 
\[
\frac{\sigma\left(D\cdot U_{2^s}\right)}{D \cdot U_{2^s}} \leq \frac{\sigma\left(U_{2^s}\right)}{U_{2^s}} \cdot \frac{\sigma\left(D\right)}{D} \ll_{g} \frac{\sigma\left(U_{2^s}\right)}{U_{2^s}}.
\] We can prove that $\frac{\sigma\left(U_{2^s}\right)}{U_{2^s}}$ is bounded by some constant in the following way. 

\begin{lemma}
\[
\sum_{n \in \mathcal{P}^{*}}{\frac{\log p}{p - 1}} = \sum_{p \in \mathcal{P}}{\frac{\log p}{p - 1}} \prod_{p \in \mathcal{P}} {\left(1-\frac{1}{p}\right)}^{-1},
\]
where $\mathcal{P}$ is a finite set of primes and $\mathcal{P}^{*}$ is the  set of natural numbers all of whose prime factors belong to $\mathcal{P}$. 
\end{lemma}\begin{proof} We will work with each term independently. To begin, we will insert the Von-Mangoldt Function.
\[
\sum_{d \in \mathcal{P}^{*}} \Lambda\left(d\right) \sum_{\substack{d \in \mathcal{P}^{*} \\ d|n}} {\frac{1}{n}} \implies \sum_{d \in \mathcal{P}^{*}} \Lambda\left(d\right) \sum_{n^\prime \in \mathcal{P}^{*}} {\frac{1}{n^\prime}}
\]
We may now insert Euler Product to obtain
\[
\prod_{p \in \mathcal{P}} {\left(1-\frac{1}{p}\right)}^{-1}.
\]
Now we will deal with $\sum_{d \in \mathcal{P}^{*}} \Lambda\left(d\right)$. This is just summing over all $d|n$ of $\Lambda\left(d\right)$, and thus we can substitute $\log p$, for some $p \in \mathbb{P}$ back into the expression. We need to then reinsert all the prime factors that do not divide $n$. By infinite geometric series, we obtain
\[
\sum_{p \in \mathcal{P}} \frac{\log p}{p - 1}
\]
and we are done. \end{proof}

\begin{lemma}
\[
\log \frac{\sigma\left(U_m\right)}{U_m} \ll_{g} \log\left(e \cdot \omega\left(m\right)\right)^2,
\]
where $m > 1$ and is an integer.
\end{lemma} \begin{proof} We know that 
\[
\frac{\sigma\left(U_m\right)}{U_m} = \prod_{p^e \pdiv U_m}\left(1 + \frac{1}{p} + \cdots + \frac{1}{p^e}\right) \leq \exp\left(\sum_{p \mid U_m} \frac{1}{p-1}\right).
\]
We also know that if $p \mid U_m$, it implies that $z\left(p\right) > 1$ where $z\left(p\right)$ is the content of $p$). If $p \nmid g-1$, then $z\left(p\right) \mid p-1$.
\[
\sum_{p \mid U_m} \frac{1}{p-1} \leq \sum_{p \mid g-1} \frac{1}{p-1} + \sum_{\substack{d \mid m \\ d > 1}} \sum_{\substack{p \mid U_d \\ p \equiv 1 \Mod{d}}} {\frac{1}{p-1}}
\]
If $d$ is a divisor of $m$ and $U_{d} < g^d$, there exists a progression $1 \Mod{d}$. And hence, it is bounded by $\frac{\log\left(g^d\right)}{log\left(d\right)}$. So, we may apply this in the following:
\[
\sum_{\substack{p \mid U_d \\ p \equiv 1 \Mod{d}}} {\frac{1}{p-1}} \leq \sum_{1 \leq k \leq d \cdot \frac{\log g}{\log d}} {\frac{1}{dk}} \leq \frac{1}{d}\left[1 + \log\left(\frac{d \log g}{\log d}\right)\right] \ll_{g} \frac{\log\left(ed\right)}{d}. 
\]
This implies that 
\[
\sum_{\substack{p \mid U_d \\ p \equiv 1 \Mod{d}}} {\frac{1}{p-1}} \ll_{g} \sum_{d \mid m} {\frac{\log\left(ed\right)}{d}}.
\]
If we let $m = q_1^{\alpha_1}\cdot q_2^{\alpha_2} \cdots q_k^{\alpha_k}$ (where $k = \omega\left(m\right)$) and $d \mid m$, then $d = q_1^{\beta_1}\cdot q_2^{\alpha_2} \cdots q_k^{\alpha_k}, 0 \leq \beta_i \leq \alpha_i$. Note that the map from $x \mapsto \frac{\log\left(ex\right)}{x}$ is decreasing for $x \geq 1$. This implies that $\log\left(\frac{\log\left(ed\right)}{d}\right) \leq \frac{\log\left(ed^\prime\right)}{d^\prime}, d^\prime = p_1^{\beta_1}\cdot p_2^{\beta_2} \cdots p_k^{\beta_k}$, where $p_i$ is the $i^{th}$ prime. We can let $\mathcal{P}, $or a finite set of primes $= {p_1, p_2, \cdots, p_k}$. 
\[
\sum_{d^\prime \in \mathcal{P}^{*}} {\frac{\log\left(ed^\prime\right)}{d^\prime}} = \left(1 + \sum_{p \in \mathcal{P}} \frac{\log p}{p-1}\right)\left(\prod_{p \in \mathcal{P}}{\left(1 + \frac{1}{p}\right)}\right) \ll \left(\log\left(e \cdot \omega\left(m\right)\right)\right)^2
\]
\textit{Note: The conclusion resulted from Merten's Estimate and Prime Number Theorem} \cite{lucapollack}. 
\end{proof}

By combining everything, we obtain what we had originally wanted. Recall that we have only proven this when $m$ is a $2^s$-perfect number.
\subsection{Multiperfect numbers and factorials}
In this section, we collect a few simple observations regarding factorials and multiperfect numbers.
\begin{prop}
The only non-negative integer value of $n$ such that $n!$ is a perfect number is $3$.
\end{prop} \begin{proof}
Clearly, $n=1$ does not work, so $n!$ must be even. We know that $$n! = 2^{p-1} \cdot\left(2^p-1\right)$$ where $p$ and $2^p-1$ are prime. Also we can assume $n \geq 3$ since $n=2$ also does not work. So $3 \mid {n! = 2^{p-1}*\left(2^p-1\right)} \implies 3 \mid {2^p-1} \implies 2^p-1 = 3 \implies p = 2 \implies n = 3$
\end{proof}
\begin{lemma}
For sufficiently large $n$, $\rad\left(n!\right) < \left(n!\right)^{\epsilon}$.
\end{lemma} \begin{proof}
By Stirling's approximation, $n!$ can be estimated as $\left(\frac{n}{e}\right)^n$. By the prime number theorem, the number of primes less than $n$ can be estimated as $\frac{n}{\log{n}}$. So,$$\rad\left(n!\right) < n^{\frac{n}{\log{n}}}.$$It suffices to show that $$n^{\frac{n}{\log{n}}} \ll \left(\frac{n}{e}\right)^n.$$ Clearly, if $n$ is sufficiently large: 

$$n^{\frac{1}{\log n}} \approx n^{\epsilon} \approx \left(\frac{n}{e}\right)^{\epsilon} \implies n^{\frac{n}{\log n}} \approx \left(\frac{n^n}{e^n}\right)^{\epsilon} \implies \rad\left(n!\right) < n!^\epsilon.$$ \end{proof}
\begin{prop}
For a multiperfect number $m$ where $\sigma(m) = km,$ there exists at most one multiperfect number that can also be expressed as $n!$ for some $n$ for any natural number $k \geq 2.$
\end{prop} 
 \begin{proof}
Let $n! = m = 2^{\alpha_1}3^{\alpha_2}\dots{p_r}^{\alpha_r}$, where $m$ is the smallest multiperfect number such that $\sigma\left(m\right)=km $ for a given $k$. Then, 
\[km = kn! = k2^{\alpha_1}3^{\alpha_2}\dots{p_r}^{\alpha_r} = \sigma\left(m\right)\]\[ = \left(1+2+\dots+2^{\alpha_1}\right)\left(1+3+\dots+2^{\alpha_1}\right)\dots\left(1+p_r+\dots+{p_r}^{\alpha_r}\right)\] Assume that there exists $n'! = m'= 2^{\beta_1}3^{\beta_2}\dots{p_c}^{\beta_c}$ , where $n'>n$ and $c\geq r$, such that $m'$ is the second smallest k-multiperfect number. We have $\beta_i \geq \alpha_i$ for every $i\leq c$, because $n'>n$. Hence,
\[k=\left(1+\frac{1}{2}+\dots+\frac{1}{2^{\alpha_1}}\right)\left(1+\frac{1}{3}+\dots+\frac{1}{3^{\alpha_2}}\right)\dots\left(1+\frac{1}{p_r}+\dots+\frac{1}{{p_r}^{\alpha_r}}\right)\]  \[=\left(1+\frac{1}{2}+\dots+\frac{1}{2^{\beta_1}}\right)\left(1+\frac{1}{3}+\dots+\frac{1}{3^{\beta_2}}\right)\dots\left(1+\frac{1}{p_c}+\dots+\frac{1}{{p_c}^{\beta_c}}\right)\] When comparing each bracket on both side of the equation, for every $i \leq r$, since $v_{p_i}\left(m\right)\leq v_{p_i}\left(m'\right), $ or in other words $\alpha_i \leq \beta_i$,
\[\left(1+\frac{1}{p_i}+\dots+\frac{1}{{p_i}^{\alpha_i}}\right)\leq \left(1+\frac{1}{p_i}+\dots+\frac{1}{{p_i}^{\beta_i}}\right).\] For every $r<j\leq c$, 
\[\left(1+\frac{1}{p_j}+\dots+\frac{1}{{p_j}^{\beta_j}}\right)>1.\] Since $n!=m\neq \left(n'\right)!=m'$,
\[\left(1+\frac{1}{2}+\dots+\frac{1}{2^{\alpha_1}}\right)\left(1+\frac{1}{3}+\dots+\frac{1}{3^{\alpha_2}}\right)\dots\left(1+\frac{1}{p_r}+\dots+\frac{1}{{p_r}^{\alpha_r}}\right)\]\[<\left(1+\frac{1}{2}+\dots+\frac{1}{2^{\beta_1}}\right)\left(1+\frac{1}{3}+\dots+\frac{1}{3^{\beta_2}}\right)\dots\left(1+\frac{1}{p_c}+\dots+\frac{1}{{p_c}^{\beta_c}}\right)\] which is impossible. Hence, for every given natural number $k$, there exists at most one multiperfect number $m$ such that $\sigma\left(m\right)=km$ as desired. \end{proof}


Next result concerns shifted factorials, that have been extensively studied by many authors throughout the years.
\begin{prop}
Assume the ABC conjecture is true. There exists finitely many $n$ such that $n!+1$ is a multi-perfect number.
\end{prop} \begin{proof}
Let $n!+1 = m$, where $m$ is a multi-perfect number. By the ABC conjecture, $$m \leq C_{\epsilon} \rad\left(1 \cdot n! \cdot m\right)^{1+\epsilon} $$ $$\implies m \leq C_{\epsilon} \rad\left(n!\right)^{1+\epsilon} \rad\left(m\right)^{1+\epsilon} \leq C_{\epsilon} n!^{\epsilon} \rad\left(n!\right)^{\epsilon} \rad\left(m\right)^{\epsilon} m^{\frac{4t+1}{4t+4}}$$ where $t$ is the largest power of $2$ that divides m. Clearly, we can now choose an $\epsilon > 0$ such that this does not hold. So, we have a contradiction, implying that there are finitely many multiperfect numbers of the form $n!+1$.
\end{proof}


The Wheatley School.

{\it Email address}: virajjayam@gmail.com

Thomas Jefferson High School for Science and Technology.

{\it Email address}: ajitkadaveru@yahoo.com

Acton Boxborough Regional High School

{\it Email address}: thinkinavi@gmail.com

The Experimental High School Attached to Beijing Normal
University.

{ \it Email address}: carrie\_zhang2001@hotmail.com

\end{document}